\documentclass[11pt,a4paper]{article}
\usepackage[top=3cm, bottom=3cm, left=3cm, right=3cm]{geometry}
\usepackage[latin1]{inputenc}
\usepackage{amsmath}
\usepackage{amsthm}
\usepackage{amsfonts}
\usepackage{amssymb}
\usepackage{graphicx}
\usepackage{enumerate}
\usepackage{enumitem}
\usepackage{hyperref}

\newtheorem{thm}{Theorem}[section]
\newtheorem{lem}{Lemma}[section]
\newtheorem{conj}{Conjecture}[section]
\newtheorem{exa}{Example}[section]

\title{Twin Primes In Quadratic Arithmetic Progressions}
\date{}
\author{N. A. Carella}

\usepackage{lipsum}
\usepackage{fancyhdr}
\begin{document}
\thispagestyle{empty}
	
\maketitle

\begin{abstract} 
A recent heuristic argument based on basic concepts in spectral analysis showed that the twin prime conjecture and a few other related primes counting problems are valid. A rigorous version of the spectral method, and a proof for the existence of infinitely many quadratic twin primes $n^{2}+1$ and $n^{2}+3$, $n \geq 1$, are proposed in this note.
\end{abstract}

\tableofcontents
\vskip .25 in
\text{\textit{Mathematics Subject Classifications}: 11A41, 11N05, 11N32.}
\vskip .15 in
\text{\textit{Keywords}: Distribution of Primes, Twin Primes Conjecture, Quadratic Primes. }
\vskip .25 in

\newpage
\section{Introduction} \label{s1}
A heuristic argument in favor of the existence of infinitely many twin primes numbers $p$ and $q = p + 2$ was devised in \cite{GP99}, and generalized to some related linear prime pairs $p$ and $ap + bq = c$ counting problems, where $a, b, c \in \mathbb{Z}$ are integers in \cite{GP06}. The heuristic argument is based on basic spectral techniques, as Ramanujan series, and the Wiener-Khintchine formula, applied to the additive correlation function $\sum_{n \leq x} \Lambda(n)\Lambda(n+2)$. Extensive background information are available in \cite{RP96}, \cite{NW00}, \cite{GD09}, \cite{PJ09}, et cetera.\\

This work proposes a rigorous version and a proof for the quadratic twin primes $n^2+1$, and $n^2+3, n\geq 1,$ counting problem. The basic idea combines elements of spectral analysis and complex functions theory to derive a rigorous proof applied to the correlation function $\sum_{n^{2}+1\leq x}\Lambda(n^{2}+1)\Lambda(n^{2}+3)$, see \cite[p. 46]{HL23}, \cite[p. 343]{NW00}, \cite[p. 406]{RP96}, et alii. The main result is as follows. \\

\begin{thm} \label{thm1.1} Let $x\geq 1$ be a large real number. Then
\begin{equation}
\sum_{n^{2}+1\leq x}\Lambda(n^{2}+1)\Lambda(n^{2}+3) =  \alpha_{2} x^{1/2}+O\left (x^{1/2} \frac{(\log \log x)^2}{\log x} \right ),
\end{equation}
for some $\alpha_{2} > 1/2$ constant.                                            
\end{thm}

This implies that there is an effective asymptotic formula for the quadratic twin primes counting function
\begin{equation}
\pi_{2}(x) =\#\{n^{2}+1,n^{2}+3\leq x \} > \frac{1}{2} \frac{x^{1/2}}{\log^2 x}+O\left (x^{1/2} \frac{(\log \log x)^2}{\log^3 x} \right ).
\end{equation}

The subsequent Sections collect the necessary elementary results in arithmetic functions, spectral analysis, complex function theory, and some ideas on convergent series. And the last Section assembles a short proof of Theorem \ref{thm1.1}.

\section{Properties of the Ramanujan Sum and Identities} \label{s2}
The symbols $\mathbb{N} = \{ 0, 1, 2, 3, \ldots \}$ and $\mathbb{Z} = \{ \ldots, -3, -2, -1, 0,  1, 2, 3, \ldots \}$ denote the subsets of integers. For $n \geq 1$, the Mobius function is defined by
\begin{equation}
\mu(n) =
\left\{
\begin{array}{ll}
(-1)^v     &n=p_{1}\cdot p_{2} \cdot \cdot \cdot p_{v},\\
0          &n\ne p_{1} \cdot p_{2}\cdot \cdot \cdot p_{v},\\
\end{array}
\right.
\end{equation}
where the $p_{i}$ are primes. The Euler totient function is defined by 
\begin{equation}
\varphi(n)=\prod_{p \mid n}(1-1/p).
\end{equation} 
And the Ramanujan sum is defined by 
\begin{equation}
c_{q}(n)=\sum_{\gcd(k,q)=1}e^{i 2 \pi k n /q}.
\end{equation} 

\subsection{Some Properties}
The arithmetic function $c_{q}(n)$ is  multiplicative but not completely multiplicative. The basic properties of $c_{q}(n)$ are discussed in \cite[p.\ 160]{AP76}, \cite[p.\ 139]{HR59}, $\cite{HW08}$, et alii. Some of these basic properties of the function $c_{q}(n)$ are listed here.\\

\begin{lem} \label{lem2.1}   {\normalfont (Identities ) }  Let $n$ and $q$ be integers. Then 
\begin{enumerate}[font=\normalfont, label=(\roman*)]
\item  $c_{q}(n)=\mu(q/d) \varphi(q)/\varphi(q/d)$, where $d = \gcd(n, q)$.
\item $c_{q}(n)=\sum_{d \mid \gcd(n,q)} \mu(q/d)d$.
\item $c_{q}(n)=\prod_{p^{\alpha} \mid \mid n}c_{p^{\alpha}(n)}$, where $p^{\alpha} \mid \mid n $ \textit{denotes the maximal prime power divisor of} $n$.
\item $c_{2m}(n)=-c_{m}(n)$ \textit{for any odd integers} $m, n \geq 1$.
\item $c_{2^{v}m}(n)=0$ \textit{for any odd integers} $m, n \geq 1$, and $v \geq 2$.
\item For any prime power
\begin{equation}
c_{p^v}(n)=\left \{\begin{array}{ll}
p^{v-1}(p-1) & \text{if } \gcd (n,p^v)=p^v, v\geq 1,\\
-p^{v-1} & \text{if } \gcd (n,p^v)=p^{v-1}, v\geq 1,\\
0 & \text{if } \gcd (n,p^v)=p^u \text{ with }0 \leq u\leq v-2. 
\end{array} \right.
\end{equation} 	
\end{enumerate}
\end{lem}

\begin{proof} (ii) implies (iv) and (v). As the integer $n$ is odd, the index of the finite sum $d | \gcd(2^{v}m, n)$ is odd. Thus, 
\begin{eqnarray}
 c_{2^{v}m}(n)&=&\sum_{d|\gcd(2^{v}m,n)} \mu(2^{v}m/d)d \nonumber \\
 &=&\mu(2^{v})\sum_{d \mid \gcd(m,n)}\mu(m/d)d  \\
 &=&\left\{
\begin{array}{ll}
-c_{m}(n)     &\text{if }v=1, \\
0          &\text{if }v\geq 2.
\end{array} 
\right. \nonumber
\end{eqnarray}
This completes the verifications of (iv) and (v) 
\end{proof}

\begin{lem} \label{lem2.2}  For any pair of integers $n,q\geq1$ the following holds.
\begin{enumerate}[font=\normalfont, label=(\roman*)]
\item For any fixed $n \geq 1$, the estimate $1 \leq | c_{q}(n)|\leq 2n \log \log n$  holds all $q \geq 1$.
\item For any fixed $q \geq 1$, the estimate $1 \leq | c_{q}(n)|\leq 2q \log \log q$  holds all $n \geq 1$.
\end{enumerate}
\end{lem}

\begin{proof} Using the formulas $c_{q}(n)=\sum_{d \mid \gcd(n,q)} \mu(q/d)d$, and $\sigma(n)=\sum_{d|n}d \leq 2n \log \log n$ it is easy to confirm that 
\begin{equation}
1 \leq | c_{q}(n)|= \left |\sum_{d \mid \gcd(n,q)} \mu(q/d)d \right  |\leq \sum_{d|n}d  \leq 2n \log \log n .
\end{equation} 
This verifies the inequality (i).  
\end{proof}

The detail on the upper bound $\sigma(n)=\sum_{d \mid n}d \leq 2n \log \log n$ appears in \cite[p. 116]{TG15} or similar references. A sharper bound  
$1 \leq | c_{q}(n)|\leq d(n)$ is claimed in \cite[p.\ 139]{HR59}. Here $d(n)=\sum_{d \mid n}1 \ll n^\varepsilon$, where $\varepsilon>0$ is an 
arbitrarily small number. \\

These and many other basic properties of $c_{q}(n)$ are discussed in \cite[p.\ 160]{AP76}, \cite[p.\ 237]{HW08}, etc. New applications and new proofs were developed in \cite{FK12} for various exponential finite sums.

\section{Mean Values of Functions} \label{s3}
A few definitions and results on the spectral analysis of arithmetic functions are discussed in this section. \\

Let $f : \mathbb{N} \longrightarrow \mathbb{C}$ be a complex valued function on the set of nonnegative integers. The mean value of a function is denoted by
\begin{equation}
M(f)=\lim_{x \rightarrow \infty} \frac{1}{x} \sum_{n \leq x}f(n),
\end{equation} 
confer \cite[p. 46]{SS94}. For $q \geq 0$, the \textit{q}th coefficient of the Ramanujan series 
\begin{equation}
f(n)=\sum_{n \geq 1}a_q c_q(n)
\end{equation}
is defined by 
\begin{equation}
a_q = \frac{1}{\varphi(q)} M(c_q(n) \overline{f}),
\end{equation} 
where $\overline{f}$ is the complex conjugate of $f$, and $\varphi$ is the Euler totient function. \\

\begin{lem} \label{lem3.1}  {\normalfont (Orthogonal relation)}. Let $n, p$ and $ r \in \mathbb{N}$ be integers, and let $c_q(n)=\sum_{\gcd(k,q)=1}e^{i 2 \pi k n /q}$. 
Then
\begin{equation}
M(c_q(n) c_r(n))=\left \{
\begin{array}{ll}
\varphi(q)     &\text{if  }q=r, \\
0          &\text{if  }q\ne r. 
\end{array}
\right .
\end{equation} 
\end{lem}

\begin{proof} Compute the mean value of the product $c_q(n)c_r(n)$:
\begin{eqnarray}
M(c_q(n)c_r(n))&=&\lim_{x \rightarrow \infty} \frac{1}{x} \sum_{n \leq x}c_q(n)c_r(n) \nonumber \\
&=&\lim_{x \rightarrow \infty} \frac{1}{x} \sum_{n \leq x,}\sum_{\gcd(u,q)=1}e^{i 2 \pi  n \frac{u}{q}} \sum_{\gcd(v,r)=1}e^{i 2 \pi  n \frac{v}{r}}\\
&=&\lim_{x \rightarrow \infty} \frac{1}{x} \sum_{n \leq x,} \sum_{\gcd(u,q)=1,}\sum_{\gcd(v,r)=1}e^{i 2 \pi  n(\frac{u}{q}- \frac{v}{r})} \nonumber.
\end{eqnarray} 
If $u/q=v/r$, then the double inner finite sum collapses to $\varphi(q)$. Otherwise, it vanishes:
\begin{eqnarray}
& &\lim_{x \rightarrow \infty} \frac{1}{x} \sum_{n \leq x,} \sum_{\gcd(u,q)=1,}\sum_{\gcd(v,r)=1}e^{i 2 \pi  n(\frac{u}{q}-\frac{v}{r})}
 \nonumber \\
 &=& \sum_{\gcd(u,q)=1,}\sum_{\gcd(v,r)=1}\lim_{x \rightarrow \infty}  \sum_{n \leq x}\frac{1}{x}e^{i 2 \pi  n(\frac{u}{q}-\frac{v}{r})}  \\
&=& \sum_{\gcd(u,q)=1}\sum_{\gcd(v,q)=1}\lim_{x \rightarrow \infty}  \frac{1}{x} \left ( \frac{1-e^{i 2 \pi  (\frac{u}{q}-\frac{v}{r})(x+1)}}{1-e^{i 2 \pi  (\frac{u}{q}-\frac{v}{r})}} \right ) \nonumber \\
&=&0 \nonumber
\end{eqnarray} 
for $u/q \ne v/r$. 
\end{proof}

Similar discussions are given in \cite[p.\ 270]{SS94}, and the literature. Extensive details of the analysis of arithmetic functions are given 
in \cite{HR59}, \cite{SS94}, \cite{LL10}, \cite{RM13} and the literature.

\section{Some Harmonic Series For Arithmetic Functions} \label{s4}
The vonMangoldt function is defined by
\begin{equation}
\Lambda(n) =
\left\{
\begin{array}{ll}
\log p&n=p^k,\\
0&n\ne p^k,k\ge1.\\
\end{array}
\right.
\end{equation} 

There are various techniques for constructing analytic approximations of the vonMangoldt function. The Ramanujan series and the Taylor series considered here are equivalent approximations. \\

\subsection{Ramanujan Series}
The Ramanujan series $f(n,s)$ at $s=1$ coincides with the ratio 
\begin{equation}
\frac{\varphi(n)}{n}\Lambda(n)=\Lambda(n)+O(\frac{\log n}{n}).
\end{equation} Hence, this series realizes an effective analytic approximation of the vonMangoldt function $\Lambda(n)$ as a function of two variables $n\geq2$ and $s\approx 1$. \\

The convergence property of the kernel series  $f(n,1)$ was proved in an earlier work, it should be noted that it is not a trivial calculation. \\

\begin{lem} \label{lem4.1}  {\normalfont  (\cite[p.\  25]{KL12})} The Ramanujan series 
\begin{equation}
\sum_{q\geq 1}\frac{\mu(q)}{ \varphi(q)}c_{q}(n)
\end{equation} 
diverges for $n = 1$ and conditionally converges for $n \geq 2$.
\end{lem}

\begin{lem} \label{lem4.2} Let $n \geq 2$ be a fixed integer. For any complex number $ s \in \mathbb{C}, \Re e(s)>1 $, the Ramanujan series
\begin{equation}
f(n,s)= \sum_{q\geq 1}\frac{\mu(q)}{q^{s-1}\varphi(q)}c_{q}(n)
\end{equation} 
is absolutely and uniformly convergent.
\end{lem}

\begin{proof} Suppose \(n\geq 1\) is a fixed integer. The inequality in Lemma \ref{lem2.2} gives 
\begin{eqnarray}
\left | \sum _{q\geq 1} \frac{\mu (q)}{q^{s-1}\varphi(q)}c_q(n) \right| &\leq& \sum _{q\geq 1} \frac{1}{q^{\sigma -1}\varphi (q)} 
\left| c_q(n) \right| \nonumber \\
&\leq & (2n \log \log n) \sum _{q\geq 1} \frac{1}{q^{\sigma -1}\varphi (q)}. \end{eqnarray}
Now, using \(q/ \log q \leq \varphi(n)\leq q \), leads to 
\begin{equation}
\left| \sum _{q\geq 1} \frac{\mu (q)}{q^{s-1}\varphi (q)}c_q(n) \right |
\leq 2n \log \log n\sum _{q\geq 1} \frac{\log q}{q^{\sigma }}.
\end{equation}
Since for large $q\geq 1$, the inequality \(\log q \leq q^{\varepsilon/2 }\) holds for any small number \(\epsilon >0\), and $\Re e(s)=\sigma \geq1+\varepsilon$,
it is clear that the series is absolutely and uniformly convergent for \(\Re e(s)\geq 1+\varepsilon\). 
\end{proof}                                                                     

Detailed analysis of Ramanujan series are given in \cite{HR59}, \cite{KL12}, \cite{LL10}, \cite{RM13}, \cite{SS94}, and similar references.  \\

\subsection{Taylor Series}
The Taylor series $f(n,s)$ at $s=1$ has the ratio`
\begin{equation}
\frac{\varphi(n)}{n}\Lambda(n)=\Lambda(n)+O(\frac{\log n}{n})
\end{equation}
as a constant term. Ergo, this series realizes an effective analytic approximation of the vonMangoldt function $\Lambda(n)$ as a function of two variables $n\geq2$ and $s\approx 1$.\\

\begin{lem} \label{lem4.3} For any integer $n \geq 2$, and any complex number $s \in \mathbb{C}$, $\Re e(s)> 1$, the Taylor series of the function $f(n,s)$ at $s = 1$ is as follows.
\begin{equation}
f(n,s)= f_0(s-1)^{\omega-1}+f_{1}(s-1)^{\omega}+f_{2}(s-1)^{\omega+1}+ \cdot\cdot \cdot,
\end{equation}
where the $kth$ coefficient $f_{k}=f_{k}(n)$ is a function of $n\geq 2$, and $\omega=\omega(n)=\#\{p \mid n\}$.
\end{lem}

\begin{proof} As the series has multiplicative coefficients, it has a product expansion:
\begin{eqnarray} \label{el43}
\sum _{q\geq 1} \frac{\mu (q)}{q^{s-1}\varphi
	(q)}c_q(n)
&=&\prod _{p\geq 2} \left(1+\sum _{k\geq 1} \frac{\mu \left(p^k\right)}{p^{k(s-1)}\varphi \left(p^k\right)}c_{p^k}(n)\right) \nonumber\\
&=&\prod _{p\geq
	2} \left(1+\frac{\mu (p)}{p^{s-1}\varphi (p)}c_p(n)\right) \nonumber\\
&=&\prod _{p\geq 2} \left(1-\frac{1}{p^{s-1}(p-1)}c_p(n)\right)\\
&=&\prod _{p|n} \left(1-\frac{1}{p^{s-1}}\right)\prod
_{p \nmid n} \left(1+\frac{1}{p^{s-1}(p-1)}\right) \nonumber \\
&=&\prod _{p|n} \left(\frac{(p-1)\left(p^{s-1}-1\right)}{p^s-p^{s-1}+1}\right)\prod _{p\geq 2} \left(1+\frac{1}{p^{s-1}(p-1)}\right) \nonumber.
\end{eqnarray}
These reductions and simplifications have used the identities \(\mu \left(p^k\right)=0,k\geq 2\), in line 2; and 
\begin{equation}
\begin{array}{ll}  & 
c_p(n)=\left \{\begin{array}{ll}
p-1 & \text{if } \gcd (n,p)=p, \\
-1 & \text{if } \gcd (n,p)=1, nonumber\\
\end{array}\right.
\end{array} 
\end{equation}
in line 4, see Lemma \ref{lem2.1}. Write the product representation as 
\begin{equation}
\prod _{p|n} \left(\frac{(p-1)\left(p^{s-1}-1\right)}{p^s-p^{s-1}+1}\right)\prod
_{p\geq 2} \left(1+\frac{1}{p^{s-1}(p-1)}\right) =\zeta (s)A(s)B(s).
\end{equation} 
The factor \(A(s)=A(n,s)\), a function of two variables of \(n\in \mathbb{N}\) and \(s\in \mathbb{C}\), has the Taylor series  
\begin{eqnarray} \label{el432}
A(s)&=&\prod
_{p|n} \left(\frac{(p-1)\left(p^{s-1}-1\right)}{p^s-p^{s-1}+1}\right) \\ &=&a_1(s-1)^{\omega}+a_2(s-1)^{\omega+1}+a_3(s-1)^{\omega+2}+\cdots \nonumber,
\end{eqnarray} 
where the \textit{k}th coefficient \(a_k=a_k(n)\) is a function of \(n\geq 2\), and $\omega=\omega(n)=\#\{p \mid n\}$ determines the multiplicity of the zero at \(s=1\). \\

\textbf{Case 1:} If \(n=p^v,v\geq 1\), where \(p\geq 2\) is prime, then  $\omega=\omega(n)=1$ and $A(s)$ has a simple zero at \(s=1\). \\

\textbf{Case 2:} If \(n\neq p^v,v\geq 1\), where
\(p\geq 2\) is prime, then $\omega=\omega(n)>1$ and $A(s)$ has a zero of multiplicity $\omega>1$ at \(s=1\). \\

The factor \(B(s)=B(n,s)\), a function of two variables of \(n\in \mathbb{N}\) and \(s\in \mathbb{C}\), has the Taylor series 
\begin{eqnarray}
B(s)&=&\prod
_{p\geq 2} \left(1+\frac{1}{p^{s-1}(p-1)}\right)\left(1-\frac{1}{p^s}\right)\\
&=&1+b_1(s-1)+b_2(s-1){}^2+b_3(s-1){}^3+\cdots \nonumber,
\end{eqnarray} 
which is analytic for \(\Re e(s)>0\). And the zeta function 
\begin{eqnarray}
\zeta (s)&=&\prod _{p\geq 2} \left(1-\frac{1}{p^s}\right)^{-1}\\
&=&\frac{1}{s-1}+\gamma _0+\gamma _1(s-1)+\gamma _2(s-1)^2+\gamma _3(s-1)^3+\cdots \nonumber ,
\end{eqnarray} 
where \(\gamma _k\) is the \textit{k}th Stieltjes constant, see \cite{KJ92}, and \cite[eq. 25.2.4]{DLMF}. Hence, the Taylor series is 
\begin{eqnarray} \label{el431}
\sum _{q\geq 1}
\frac{\mu (q)}{q^{s-1}\varphi (q)}c_q(n)
&=&\zeta (s)A(s)B(s) \\
&=&\left(\frac{1}{s-1}+\gamma _0+\gamma _1(s-1)+\gamma _2(s-1)^2+\gamma _3(s-1)^3+\cdots
\right) \nonumber\\
&\qquad \times& \left(a_1(s-1)^\omega+a_2(s-1)^{\omega+1}+a_3(s-1)^{\omega+2}+\cdots \right) \nonumber\\
&\qquad \times&\left(1+b_1(s-1)+b_2(s-1)^2+b_3(s-1)^3+\cdots  \right) \nonumber\\
&=&a_1(s-1)^{\omega-1}+\left(a_1\gamma _0+a_1b_1+a_2 \right)(s-1)^{\omega} \nonumber\\ 
&\qquad +&\left(a_1b_2+a_1\gamma _1+a_2\gamma _0+a_3+a_1b_1\gamma _0\right)(s-1)^{\omega+1}+\cdots \nonumber\\
&=&f_0(s-1)^{\omega-1}+f_{1}(s-1)^{\omega}+f_{2}(s-1)^{\omega+1}+ \cdots \nonumber,
\end{eqnarray}
for \(s\in \mathbb{C} \) such that \(\Re e(s) >1\). In synopsis, the Ramanujan series and Taylor series are equivalent: 
\begin{equation}
\sum _{q\geq 1} \frac{\mu (q)}{q^{s-1}\varphi (q)}c_q(n)=f_0(s-1)^{\omega-1} (n)+f_1(s-1)^{\omega}+f_2(s-1)^{\omega+1}+\cdots .
\end{equation} 
In particular, evaluating the product representation or the Taylor series or Ramanujan series at \(s=1\) shows that 
\begin{equation}
 \begin{array}{ll}
& 
f(n,1)=
\left \{\begin{array}{ll}
a_1 & \text{   if  } n=p^k,k\geq 1; \\
0 & \text{   if  } n\neq p^k, k\geq 1; \\
\infty  & \text{   if  } n=1 . \\
\end{array}\right .  \\
\end{array} 
\end{equation} 
These complete the proof. 
\end{proof}

The first few coefficients\\

$f_0=a_1$, \hskip 3 in $f_1=a_1\gamma _0+a_1b_1+a_2$, \\ 
$f_2=a_1b_2+a_1\gamma _1+a_2\gamma _0+a_3+a_1b_1\gamma _0$,\\
 
are determined from the derivatives of \(A(s)\) and $B(s)$. In particular, the first derivative is 
\begin{equation}
\frac{d}{d s}A(s)=\prod _{p|n} \left(\frac{(p-1)p^{s-1}}{p^s-p^{s-1}+1}-\frac{(p-1)^2\left(p^{s-1}-1\right)p^{s-1}}{\left(p^s-p^{s-1}+1\right)^2}\right)
\log  p .
\end{equation}
Accordingly, the first coefficient $a_1=A'(1)$ is written as
\begin{equation} \label{el433}
\left. a_1=A^{'}(s) \right|_{s=1}=\prod_{p|n} \left (\frac{p-1}{p}\log p \right ).
\end{equation} 
Similarly, the second coefficient $a_2=A^{''}(1)$ is written as 
\begin{equation} \label{el434}
\left. a_2=A^{''}(s) \right|_{s=1}=\prod_{p|n} \left (\frac{p-1}{p^2}\log^2 p \right ).
\end{equation} 

\begin{exa}  { \normalfont This illustrates the effect of the decomposition of the integer $n\geq 1$ on the multiplicity of the zero of the function $f(n,s)$ and some details on the shapes of the coefficients.
\begin{enumerate}
\item If $n=p^k,k\geq1$ is a prime power, which is equivalent to $\omega(n)=1$, then
\begin{equation}
f(n,s)=f_0+f_1(s-1)+f_2(s-1)^{2}+\cdots,
\end{equation}  
where 
\begin{equation}\label{9999999}
f_0=\prod_{p|n} \left (\frac{p-1}{p}\log p \right ) =\frac{\varphi (n)}{n}\Lambda
(n)
\end{equation}
is the most important coefficient. If $|s-1|$ is small, then the other coefficients $f_1,f_2, \ldots$ are absorbed in the error term.

\item If $n=p_1^a p_2^b,a,b\geq 1$ is a product of two prime powers, which is equivalent to $\omega(n)=2$, then
\begin{equation}
f(n,s)=f_0(s-1)+f_1(s-1)^{2}+f_2(s-1)^{3}+\cdots,
\end{equation}
where 
\begin{equation}\label{999999}
f_0=\prod_{p|n} \left (\frac{p-1}{p}\log p \right ) \ne\frac{\varphi (n)}{n}\Lambda(n).
\end{equation} 
If $|s-1|$ is small, then all the coefficients $f_0,f_1,f_2, \ldots$ are absorbed in the error term.
\end{enumerate} 
}
\end{exa}

These information are used in the next result to assemble an effective approximation of the vonMangoldt function: 
\begin{equation}
\begin{array}{ll}
& 
f(n,s)=
\left \{\begin{array}{ll}
\displaystyle \frac{ \varphi(n)}{n} \Lambda(n)+O(\varepsilon) & \text{   if  } n=p^k,k\geq 1, \text{ and } 0<|s-1|<\delta; \\
\displaystyle O(\varepsilon) & \text{   if  } n\neq p^k, k\geq 1, \text{ and } 0<|s-1|<\delta. \\
\end{array}\right .  \\
\end{array} 
\end{equation} 
The details and the proof are given below. \\
                                 
\begin{lem} \label{lem4.4}  Let $\varepsilon >0$ be an arbitrarily small number, and let $n \geq 2$. Then, the Taylor series satisfies
\begin{enumerate} [font=\normalfont, label=(\roman*)]
\item [(i)] $\displaystyle \left | f(n,s)-\frac{ \varphi(n)}{n} \Lambda(n) \right | \leq \varepsilon, $ 
\item $\displaystyle f(n,s)=\frac{ \varphi(n)}{n} \Lambda(n) +O(\varepsilon), $
\item $\displaystyle f(n,s)=\frac{ \varphi(n)}{n} \Lambda(n) +f^{'}(n,s_0)(s-1), $ \textit{for some real number} $0<|s_0-1|<\delta$, 
\end{enumerate}	
where $0< |s-1|< \delta$, and a constant $\delta=c \varepsilon>0$, with $c>0$ constant.
\end{lem}

\begin{proof} For $n > 1$, the Taylor series $f(n,s)$ is a continuous function of a complex number $s \in \mathbb{C}, \Re e(s) = \sigma>1$. It has continuous and absolutely convergent derivatives
\begin{equation}
f^{(k)}(n,s)=(-1)^k \sum _{q\geq 1} \frac{\mu (q)\log^k q}{q^{s-1}\varphi (q)}c_q(n),
\end{equation}
 on the half plane $\Re e(s)>1$ for $k \geq1$. By the definition of continuity, for any $\varepsilon>0$ arbitrarily small number, there exists a small real number $\delta>0$ such that
\begin{equation}
\left | f(n,s)-f(n,1) \right | < \varepsilon
\end{equation}
for all $s \in \mathbb{C}$ such that $0< |s-1|< \delta$. The relation $\delta=c \varepsilon>0$, with $c>0$ constant, is a direct consequence of uniform continuity, confer Lemma \ref{lem4.2}. 

The statement (iii) follows from the Mean Value Theorem:
\begin{equation}
 (s-1) f^{'}(n,s_0) =f(n,s)-f(n,1)
\end{equation}
for some fixed real number $0<|s_0-1|<\delta$. Moreover, choice $\delta=\varepsilon/(2f^{'}(n,s_0))$ is sufficient. 
\end{proof}

The uniform convergence of the derivative $f^{'}(n,s)$ for $\Re e(s)>1$, and the Mean Value Theorem imply that the product $|(s-1)f^{'}(n,s)|< \varepsilon $ is uniformly bounded for all $n\geq 1$ and $\Re e(s)>1$. Accordingly, statements (ii) and (iii) in Lemma \ref{lem4.4} are equivalent, and can be interchanged in the proof of Lemma \ref{lem6.1}.\\

\section{The Wiener Khintchine Formula} \label{s5}
A basic result in Fourier analysis on the correlation function of a pair of continuous functions, known as the Wiener-Khintchine formula, is an important part of this work.  
\\

\begin{thm} \label{thm5.1}  If $f(s)=\sum_{n\geq1} a_{q} c_{q}(n) $ is an absolutely convergent Ramanujan series of a number theoretical function $f : \mathbb{N}\longrightarrow \mathbb{C}$, and $m \geq 1$ is a fixed integer, then, 
\begin{eqnarray}
\lim_{x\rightarrow \infty} \frac{1}{x} \sum_{n\leq x}f(n,s)f(n+m,s)=\sum_{q\geq 1} a_{q}^{2}c_{q}(m).
\end{eqnarray}
\end{thm}

\begin{proof} Compute the mean value of the product $f(n,s)f(n+m,s)$ as stated in Section \ref{s3}. This requires the Ramanujan series, and the exponential sum $c_q(n)=\sum _{\gcd(u,q)=1} e^{i2\pi nu/q} $ as demonstrated here.
\begin{eqnarray}
&&\lim_{x\longrightarrow \infty } \frac{1}{x}\sum _{n\leq x} f(n)f(n+m) \nonumber \\
&=&\lim_{x\longrightarrow \infty } \frac{1}{x}\sum _{n\leq x} \left(\sum _{q\geq
	1} a_qc_q(n)\right)\left(\sum _{r\geq 1} a_rc_r(n+m)\right)\\
&=&\lim_{x\longrightarrow \infty } \frac{1}{x}\sum _{n\leq x} \left(\sum _{q\geq 1} a_q\sum _{\gcd(u,q)=1} e^{\text{i2$\pi $} n \frac{u
	}{q}}\right)\left(\sum _{r\geq 1} a_r\sum _{\gcd (v,r)=1} e^{\text{i2$\pi $} (n+m) \frac{v }{r}}\right) \nonumber.
\end{eqnarray}
The last equation follows from the orthogonal mean value relation, Lemma \ref{lem3.1}. Next, use the absolute convergence to exchange the order of summation:
\begin{eqnarray}
&& \lim_{x\rightarrow \infty } \frac{1}{x}\sum _{n\leq x} f(n)f(n+m)\nonumber \\
&=&\sum _{q\geq 1} \sum _{r\geq 1} a_qa_r\lim_{x\rightarrow \infty
} \frac{1}{x}\sum _{n\leq x} \sum _{\gcd (u,q)=1} e^{i2\pi n \frac{u }{q}}\sum_{\gcd (v,r)=1} e^{-i2\pi 
(n+m) \frac{v}{r}} \nonumber \\
&=&\sum_{q\geq 1} \sum_{r\geq 1} a_qa_r \sum _{\gcd(v,r)=1} e^{i2\pi m\frac{v }{r}}\lim_{x\rightarrow \infty}\sum _{n\leq x} \frac{1}{x}\sum_{\gcd (u,q)=1} e^{-i2\pi
n\left (\frac{u}{q}- \frac{v }{r}\right)} \nonumber \\
&=&
\sum _{q\geq 1} a_qa_qc_q(m).
\end{eqnarray}
Reapply Lemma \ref{lem3.1} to complete the proof.  
\end{proof}

The Wiener-Khintchine Theorem is related to the Convolution Theorem in Fourier analysis. Other related topics are Parseval formula, Poisson summation formula, et cetera, confer the literature for more details. Some earliest works are given in \cite{WN66}.
\\

A new result on the error term associated with the correlation of a pair of functions complements the well established Wiener-Khintchine formula. This is an essential part for applications in number theory.

\begin{thm} \label{thm5.2}  { \normalfont (\cite{CM15}) } Let $f$ and $g$ be two arithmetic functions with absolutely convergent Ramanujan series
\begin{equation}
f(s)=\sum_{q\geq1} a_{q} c_{q}(n) \qquad    \mathrm{and} \qquad    g(s)=\sum_{q\geq1} b_{q} c_{q}(n),
\end{equation}
where the coefficients satisfy $a_{q},b_{q}=O(q^{-1-\beta})$, with $0 < \beta < 1$, and let $m\geq1$ be a fixed integer. Then,
\begin{equation}
\sum_{n\leq x} f(n,s)f(n+m,s) 
=x \sum_{q\geq1}a_{q}b_{q}c_{q}(m)+O\left(x^{1-\beta}(\log x)^{4-2\beta}\right) .
\end{equation}
\end{thm}

For $\beta>1$, it readily follows that the error term is significantly better. The proof of this result is lengthy, and should studied in the original source.\\

\section{Approximations For The Correlation Function} \label{s6}
Analytic formulae for approximating the autocorrelation function $\sum_{n \leq x} f(n,s) f(n+m,s)$ are considered here. Two different ways of approximating the correlation function of $f(n,s)$ are demonstrated.\\

\subsection{Taylor Domain}
The Taylor series is used to derive an analytic formula for approximating the correlation function of $f(n,s)$ in terms of the function $\Lambda(n)$ and its shift $\Lambda(n+m)$.

\begin{lem} \label{lem6.1} Let $x\geq 1$ be a large number, and let $m\geq 1$ be a fixed integer. Then, the correlation function satisfies
\begin{equation}
\sum_{n\leq x} f(n,s)f(n+m,s) 
= \sum_{n\leq x}\Lambda(n)\Lambda(n+m)+O \left (x\frac{(\log\log x)^2}{\log x} \right ),
\end{equation}
where $0<|s-1|<\delta$, and $0< \delta \ll ({\log \log x)^2}/{\log x}$.
\end{lem}

\begin{proof} Let $x \geq 1$ be a large number, and let  $\varepsilon=(\log \log x)^2/\log x>0$. By Lemma \ref{lem4.4}, the correlation function can be rewritten as 
\begin{eqnarray}  \label{el40}
&& \sum_{n\leq x} f(n,s)f(n+m,s)  \\
&=& \sum_{n\leq x}\left (\frac{\varphi (n)}{n}\Lambda
(n)+O(\varepsilon) \right )\left (\frac{\varphi (n+m)}{n+m}\Lambda(n+m)+O(\varepsilon) \right ) \nonumber\\
&=&\sum_{n\leq x}\frac{\varphi (n)}{n}\frac{\varphi (n+m)}{n+m} \Lambda(n)\Lambda
(n+m) \nonumber \\ && +O\left(\varepsilon\sum_{n\leq x} \left (\frac{\varphi (n)}{n}\Lambda(n)+\frac{\varphi (n+m)}{n+m}\Lambda\nonumber
(n+m) +\varepsilon\right ) \right ) \\
&=&\text{Main Term+Error Term} \nonumber,
\end{eqnarray}
where $0<|s-1|<\delta$ for some small number $0< \delta \ll (\log \log x)^2/\log x=\varepsilon$. Now, replace
\begin{equation}
\frac{\varphi (n)}{n}\Lambda
(n)=\Lambda
(n)+O \left (\frac{\log n}{n}\right),
\end{equation}
and 
\begin{equation}
\sum_{n\leq x}\frac{\Lambda(n)}{n}=\log x +O(1),
\end{equation}
see \cite[p. 88]{AP76}, and \cite[p. 348]{HW08}, into the main term to reduce it to
\begin{eqnarray} \label{el41}
&& \sum_{n \leq x} \left ( \frac{\varphi (n)}{n} \Lambda(n) \right ) \left (  \frac{\varphi (n+m)}{n+m}  \Lambda(n+m) \right )   \\
&=& \sum_{n \leq x}\left (\Lambda(n)+O\left (\frac{\log n}{n}\right) \right ) \left(\Lambda(n+m)+O\left (\frac{\log n}{n} \right ) \right ) \nonumber \\
&=& \sum_{n \leq x}\Lambda(n)\Lambda(n+m) +O\left (\sum_{n \leq x}\left (\frac{\Lambda(n)\log n}{n}+\frac{\Lambda(n+m)\log n}{n}+ \frac{\log^2 n}{n^2}\right ) \right ) \nonumber \\
&= & \sum_{n \leq x} \Lambda(n) \Lambda(n+m)+O\left ( \log^2 x \right ) \nonumber.
\end{eqnarray}
Note that for fixed $m\geq 1$, and large $n\geq 2$, the expression
\begin{equation} 
\frac{\varphi(n+m)}{n+m}\Lambda(n+m)=\Lambda(n+m)+
O(\frac{\log n}{n} )       
\end{equation}\\
was used in the previous calculations. Similarly, apply the Prime Number Theorem
\begin{equation}
\sum_{n\leq x}\Lambda(n)=x +O\left (x e^{-c \sqrt{\log x}} \right),
\end{equation}
where $c>0$ is an absolute constant, refer to \cite[p. 179]{MV07}, \cite[p. 277]{TG15}, into the error term to reduce it to
\begin{eqnarray} \label{el42}
&& \sum_{n\leq x} \left (\frac{\varphi (n)}{n}\Lambda(n)+\frac{\varphi (n+m)}{n+m}\Lambda(n+m)+\varepsilon \right ) \nonumber \\
&=&\sum_{n\leq x}\left (\Lambda(n)+O\left(\frac{\log n}{n}\right) \right )  +\sum_{n\leq x}\left (\Lambda(n+m)+O\left(\frac{\log n}{n}\right)\right )+\varepsilon x \nonumber\\
&=&2x+\varepsilon x+O\left (x e^{-c \sqrt{\log x}} \right) \nonumber\\&=&O(x) .
\end{eqnarray}
Replace (\ref{el41}) and (\ref{el42}) back into (\ref{el40}) to obtain
\begin{eqnarray}  \label{el44}
\sum_{n\leq x} f(n,s)f(n+m,s) 
&=& \text{Main Term+Error Term} \\
&=& \sum_{n \leq x} \Lambda(n) \Lambda(n+m)+O\left ( \log^2 x \right )+O(\varepsilon x) \nonumber \\
&=& \sum_{n \leq x} \Lambda(n) \Lambda(n+m)+O \left (x\frac{(\log\log x)^2}{\log x} \right ) \nonumber.
\end{eqnarray}
These confirm the claim.     
\end{proof}

\subsection{Ramanujan Domain}
The Ramanujan series is used to derive an analytic formula for approximating the correlation function of $f(n,s)$ in terms of the constant $\alpha_m=\sum_{q\geq1}\left | \frac{\mu(q)}{q^{s-1}\varphi(q)} \right |^{2}c_{q}(m) \geq 0$.\\

\begin{lem} \label{lem6.2} Let $x\geq 2$ be a large number, and let $m\geq 1$ be a fixed integer. Then, the correlation function satisfies
\begin{equation}
\sum_{n\leq x} f(n,s)f(n+m,s) 
=\alpha_{m}x+O \left(x\frac{(\log\log x)^{2})}{\log x}\right),
\end{equation}
where $0<|s-1|<\delta$, with $0< \delta \ll ({\log \log x)^2}/{\log x}$.
\end{lem}

\begin{proof} Applying Theorem \ref{thm5.1},  (Weiner-Khintchine formula), yields the spectrum (mean value of the correlation function of $f(n,s)$): 
\begin{equation}
\lim_{x\rightarrow \infty} \frac{1}{x} \sum_{n \leq x}f(n,s)f(n+m,s)=\sum_{q\geq 1} \left | \frac{\mu(q)}{q^{s-1}\varphi(q)} \right |^{2}c_{q}(m)=\alpha_{m}.
\end{equation}
For fixed $m \geq 1$, the nonnegative constant $\alpha_{m} \geq 0$ determines the density of the subset of prime pairs $\{ p, p + m: p\text{ is prime} \} \subset \mathbb{P}$ in the set of primes $\mathbb{P} = \{ 2, 3, 5,\ldots \}$, see Theorem \ref{thm7.1} for more details.\\

By Lemma \ref{lem4.2}, the \textit{q}th coefficient satisfies
\begin{equation}
\left | a_{q}(s)\right | = \left | \frac{\mu(q)}{q^{s-1}\varphi(q)} \right | \leq\frac{\log q}{q^{\sigma}}\leq\frac{1}{q^{1+\beta}},
\end{equation} 
where $\beta=c\varepsilon/2=c({\log \log x)^2}/{ 2 \log x}>0$.\\

Applying Theorem \ref{thm5.2}, to the correlation function of the Ramanujan series of $f(n,s)$ returns the asymptotic formula
\begin{eqnarray} 
\sum_{n\leq x} f(n,s)f(n+m,s) 
&=&x \sum_{q\geq1}\left | \frac{\mu(q)}{q^{s-1}\varphi(q)} \right |^{2}c_{q}(m)+O(x^{1-\beta}(\log x)^{4-2\beta})\\
&=& \alpha_{m} x+O \left (x\frac{(\log\log x)^2}{\log x} \right ) \nonumber,
\end{eqnarray}
where $\Re e(s)=1+c(\log\log x)^{2} / \log x=1+2\beta>1$ as required by Theorem \ref{thm5.2}, and 
\begin{equation}
x^{1-\beta}(\log x)^{4-2\beta}=x^{1-\frac{(\log\log x)^2}{ 2\log x}} (\log x)^{4} \leq x \frac{(\log\log x)^2}{\log x}.
\end{equation}
This holds for all complex numbers $s \in \mathbb{C}$ for which $0<|s-1|<\delta$, with $0< \delta \ll \varepsilon\leq (\log \log x)^{2}/{\log x}$, see Lemma \ref{lem4.4}.  
\end{proof}

\section{Spectrum of $f(n,s)$} \label{s7}
For a fixed integer $m\geq1$, the constant $\alpha_m$ approximates the density of the subset of prime pairs $\{p,p+m\}\subset \mathbb{P}$ in the set of primes $\mathbb{P}=\{2,3,5,7, \ldots\}$. Some of the basic properties of the densities associated with the subset of primes pairs are explained in this Section.\\

\subsection{The Function $F(\sigma)$}
The coefficients of the series $f(n,s)$ are multiplicative. This property enable simple derivation and simple expressions for the constants, which approximate or determine the densities of the subsets for the various primes counting problems. Furthermore, the spectrum analysis in Theorem \ref{thm7.1} provides a detailed view of certain properties of the densities  as a function of $m \geq 1$ and $\sigma>1/2$.\\

\begin{thm} \label{thm7.1}  Let $m \geq1$ be a fixed integer, and let $\Re e(s)=\sigma>1$ be a complex number. Then, the spectrum of the correlation function $f(n,s)$ has a product decomposition as
\begin{equation}
F(\sigma)=\prod_{p|m} \left (1+\frac{1}{p^{2\sigma-2}(p-1)}\right ) \prod_{p\nmid m} \left (1-\frac{1}{p^{2\sigma-2}(p-1)^{2}}\right ) .
\end{equation} 
In particular, the followings statements hold.
\begin{enumerate} 
\item If $m=2k+1$ is a fixed odd integer, then the density of prime pairs is $F(1)=0$. This implies that the number of prime pairs is finite or a subset of zero density.
\item If $m=2k\geq 2$ is a fixed even integer, and $\Re e(s)=\sigma \geq 1$, then the density of prime pairs is $F(\sigma)\geq 1/2$. This implies that the number of prime pairs is infinite.
\end{enumerate}
\end{thm}

\begin{proof} By Lemma \ref{lem4.2}, the function $f(n,s)$ has an absolutely and uniformly convergent Ramanujan series for $\Re e(s)=\sigma>1$. Applying Theorem \ref{thm5.1} (Weiner-Khintchine formula) yields the spectrum
\begin{eqnarray}
F(\sigma)&=& \lim_{x \rightarrow \infty}\frac{1}{x} \sum_{n \leq x} f(n,s)f(n+m,\overline{s}) \nonumber\\
&=&\lim_{x \rightarrow \infty}\frac{1}{x} \sum_{n \leq x} \left ( \sum_{q\geq 1} \frac{\mu(q)}{q^{s-1}\varphi(q)} c_q(n)\right)\left ( \sum_{r\geq 1} \frac{\mu(r)}{r^{s-1}\varphi(r)} c_r(n+m)\right)    \\
&=&\sum_{q\geq 1} \left | \frac{\mu(q)}{q^{s-1}\varphi(q)} \right |^{2} c_{q}(m) \nonumber.
\end{eqnarray}
Using the multiplicative property of the coefficients returns the conversion to a product
\begin{equation}
F(\sigma)=\sum_{q\geq 1} \left | \frac{\mu(q)}{q^{s-1}\varphi(q)} \right |^{2} c_{q}(m)=\prod_{p\geq 2} \left (1+\frac{\mu(p)^2}{p^{2\sigma-2}\varphi(p)^2} \overline{c}_{p}(m)\right ) .
\end{equation}
If $m=2k+1$ is odd, then the factor for $q=2$ is
\begin{equation}
1+\frac{\mu(p)^2}{p^{2\sigma-2}\varphi(p)^2} \overline{c}_{p}(m)=1-\frac{1}{2^{2\sigma-2}}.
\end{equation}
This follows from Lemma \ref{lem2.1}, that is, $c_q(m)=\overline{c_p(m)}$ and \\
\begin{equation}
c_{p}(m)=\left \{\begin{array}{ll}
p-1 & \text{if } \gcd (m,p)=p, \\
-1 & \text{if } \gcd (m,p)=1. 
\end{array}\right.
\end{equation} 
Consequently, the spectrum $F(\sigma)$ vanishes at $s=\sigma=1$, this proves (i). While for even $m=2k\geq 2$, it reduces to
\begin{equation}
F(\sigma)=\prod_{p |2k} \left (1+\frac{1}{p^{2\sigma-2}(p-1)}\right ) \prod_{p \not 2k} \left (1-\frac{1}{p^{2\sigma-2}(p-1)^{2}}\right ) .
\end{equation}
Lastly, it clear that the product converges and $F(\sigma)>0$ for $\Re e(s)=\sigma>1$. 
\end{proof}

\vskip .5 in 
\subsection{The Limits $F(1)$}
The constant $\alpha_k>0$ in Theorem \ref{thm1.1} is given by a fixed value of the spectrum function, that is, $\alpha_k=F(\sigma),\sigma>1$ fixed. For $1<\sigma$, it is clear that the product converges, and $\alpha_k>1/2$. This demonstrates that the density of prime pairs $p,p+2k$ increases as the parameter $\sigma\rightarrow 1^{+}$. In particular, as $ \sigma=1+2\beta=1+2c\varepsilon>1$ decreases, the limit of the spectrum is
\begin{eqnarray}
A_{k}&=&\lim_{\varepsilon\rightarrow 0} F(1+c\varepsilon) \nonumber\\
&=&\lim_{\varepsilon\rightarrow 0}\prod_{p| 2k} \left (1+\frac{1}{p^{2c\varepsilon}(p-1)}\right ) \prod_{p \nmid 2k} \left (1-\frac{1}{p^{2c\varepsilon}(p-1)^{2}}\right )\\
&=& \prod_{p| 2k} \left (1+\frac{1}{p-1}\right ) \prod_{p\nmid 2k} \left (1-\frac{1}{(p-1)^{2}}\right ) \nonumber,
\end{eqnarray}
refer to \cite{RP96}, \cite{NW00}, \cite{GG05}, et alii, for more information. This is precisely the constant
\begin{equation}
A_{k}=\prod_{p|2k} \left (1+\frac{1}{p-1}\right ) \prod_{p\nmid 2k} \left (1-\frac{1}{(p-1)^{2}}\right )\geq \alpha_{k}. 
\end{equation}
associated with the density of the subset of prime pairs $\{p,p+2k\}$ in the set of primes $\mathbb{P} = \{ 2, 3, 5, 7, 11, \ldots \}$. For example, for $2k = 2$, this is the twin prime constant:
\begin{equation}
A_{2}=2 \prod_{p\geq3} \left (1-\frac{1}{(p-1)^2}\right )=1.32058148001344\ldots \:.
\end{equation}\\

\section{Quadratic Twin Primes} \label{s9}
The quadratic twin primes conjecture calls for an infinite subset of prime pairs $n^2+1, n^2+3$ as the integer $n \geq 1$ tends to infinity. Background information, and various descriptions of this primes counting problem are given in \cite[p. 46]{HL23}, \cite[p. 343]{NW00}, \cite[p. 408]{RP96},  \cite{PJ09}, et alii.
\\

The original derivation of the heuristic, based on the circle method, which appears in \cite[p. 62]{HL23}, concluded in the followings observation.\\ 

\begin{conj} {\normalfont (Quadratic Twin Primes Conjecture)} There are infinitely many twin primes $n^2+1, n^2+3$ as $n \to \infty$. Moreover, the counting function has the asymptotic formula
\begin{equation}
\pi _{f_2}(x)=6\frac{x^{1/2}}{\log ^2 x}\prod _{p\geq 5} \left(\frac{p(p-v_p(f))}{(p-1)^2}\right)+O\left(\frac{x^{1/2}}{\log ^3 x}\right),
\end{equation}
where \(v_p(f)=0,2,4\) according to the occurrence of $0,1$ or $2$ quadratic residues $\text{mod } p$ in the set $\{-1,-3\}$.
\end{conj}

The constant arises from the singular series
\begin{equation}
\mathfrak{G}(f)=\prod_{p\geq 2} \left (1-\frac{v_p(f)}{p} \right ) \left (1-\frac{1}{p} \right )^{-2}=6 \prod _{p\geq 5} \left(\frac{p(p-v_p(f))}{(p-1)^2}\right),
\end{equation}
where $v_p(f)$ is the number of root in the congruence $(x^2+1)(x^2+3)\equiv 0 \text{ mod }p$ for $p \geq 2$. In this case there is an exact formula
\begin{equation}
v_p(f) =\left\{
\begin{array}{ll}
0     &\text{if }p \equiv 11 \text{ mod } 12,\\
2          &\text{if }p \equiv 5,7 \text{ mod } 12,\\
4          &\text{if }p \equiv 1 \text{ mod } 12.\\
\end{array}
\right.
\end{equation}

The heuristic based on probability, yields the Gaussian type analytic formula
\begin{equation}
\pi _{f_2}(x)=C_{2} \int_{2}^{x^{1/2}} \frac{1}{ \sqrt t \log^2 t} dt+O\left(\frac{x^{1/2}}{\log ^3 x}\right),
\end{equation}
where $C_{2}=\mathfrak{G}(f)$ is the same constant as above. The deterministic approach is based on the weighted prime pairs counting function
\begin{equation}
\sum_{n^2+1\leq x}\Lambda(n^2+1)\Lambda(n^2+3),
\end{equation}
which is basically an extended version of the Chebyshev method for counting primes.\\

\subsection{The Proof of Theorem 1.1}
A lower bound for the correlation function $\sum_{n^{2}+1\leq x}\Lambda(n^{2}+1)\Lambda(n^{2}+3)$ will be derived here. It exploits the duality or equivalence of the Ramanujan series domain and the Taylor series domain to compute the correlation function of in two distinct ways. \\

\begin{proof} {\normalfont (Theorem \ref{thm1.1})} Let $x \geq 1$ be a large number, and let $\varepsilon= ({\log \log x)^2}/{\log x}>0$. By Lemma \ref{lem2.1}, the function $f(n^{2}+1,s)$ has an absolutely, and uniformly convergent Ramanujan series for 
$\Re e(s) = \sigma \geq 1+c({\log \log x)^2}/{\log x}>1$, with $0<c<1$ constant. \\

By Lemma \ref{lem6.1}, the correlation function of the Taylor series of $f(n^{2}+1,s)$ has the asymptotic formula
\begin{equation} \label{el90}
\sum_{n^{2}+1\leq x} f(n^{2}+1,s)f(n^{2}+3,s) 
= \sum_{n^{2}+1\leq x}\Lambda(n^{2}+1)\Lambda(n^{2}+3)+O \left (x^{1/2}\frac{(\log\log x)^2}{\log x} \right ),
\end{equation}
where $0<|s-1|<\delta$, with $0< \delta \ll \varepsilon=({\log \log x)^2}/{\log x}$, see Lemma \ref{lem4.4}.\\

Applying Lemma \ref{lem6.2}, to the correlation function of the Ramanujan series of $f(n^{2}+1,s)$ returns the asymptotic formula
\begin{equation} \label{el91}
\sum_{n^{2}+1\leq x} f(n^{2}+1,s)f(n^{2}+3,s) 
=\alpha_{2} x^{1/2}+O \left (x^{1/2}\frac{(\log\log x)^2}{\log x} \right ),
\end{equation}
where $\Re e(s)=1+c(\log\log x)^{2}/\log x=1+2\beta>1$ as required by Theorem \ref{lem6.2}. \\ 

Since the constant $\alpha_k>1/2$ for $\Re e(s)>1$, see Theorem \ref{thm7.1}, combining (\ref{el90}) and (\ref{el91}) confirms the claim: 
\begin{eqnarray}
& &\sum_{n^2+1\leq x}\Lambda(n^{2}+1)\Lambda(n^{2}+3) +O \left(x^{1/2}\frac{(\log\log x)^{2}}{\log x}\right)  \\
&=& \sum_{n^{2}+1\leq x} f(n^{2}+1,s)f(n^{2}+3,s)  \nonumber\\
&=& \alpha_{2}x^{1/2}+O \left(x^{1/2}\frac{(\log\log x)^{2}}{\log x}\right).
\end{eqnarray}
Therefore, 
\begin{eqnarray}
\sum_{n^2+1\leq x}\Lambda(n^{2}+1)\Lambda(n^{2}+3)
&=& \alpha_{2}x^{1/2}+O \left(x^{1/2}\frac{(\log\log x)^{2}}{\log x}\right)\\
&\geq& (1/2)x^{1/2}+O \left(x^{1/2}\frac{(\log\log x)^{2}}{\log x}\right) \nonumber
\end{eqnarray}
is a lower bound of the correlation function.   \end{proof}

The nonnegative constant
\begin{equation}
\alpha_2=\sum_{q\geq 1} \left | \frac{\mu(q)}{q^{s-1}\varphi(q)} \right |^2c_{q}(2)\geq \frac{1}{2}
\end{equation} 
for $\sigma>1$ shows that the subset of quadratic twin primes $\mathbb{T}_2=\{ n^{2}+1, n^{2}+3: n\geq1 \}$ has positive density 
\begin{equation}
\delta(\mathbb{T}_2)=\lim_{x \to \infty} \frac{\#\{ p=n^{2}+1\leq x: p \text{ and } p+2 \text{ are primes} \}}{\#\{ p=n^{2}+1\leq x: p \text{ is prime} \}}>0
\end{equation}

in the set of quadratic primes $\mathbb{P}_{2} = \{ 5, 17, 37,\ldots, n^{2}+1, \ldots :n\geq1\}$.
\\

\subsection{Numerical Data}
A small numerical experiment was conducted to show the abundance of quadratic twin primes and to estimate the constant. The counting function is defined by
\begin{equation}
\pi _{f_2}(x)=\#\{p=n^2+1\leq x: p \text{ and } p+2 \text{ are primes}\}.
\end{equation}
The data shows a converging stable constant. \\

\begin{table}
\begin{center}
 \begin{tabular}{||c |c| c| c|c|c|c|} 
 \hline
$ x$ & $\pi_{f_2}(x)$  & $\pi_{f_2}(x)/(x^{1/2}/\log x)$ & $ x$ & $\pi_{f_2}(x)$  & $\pi_{f_2}(x)/(x^{1/2}/\log x)$ \\ [1ex] 
 \hline\hline
$ 10^3$ & 4 &6.03579 &$ 10^9$&278&3.77538\\ 
 \hline
$10^5$&12 &5.02982 &$ 10^{10}$&689 &3.65301 \\
 \hline
 $10^6$ &28 &5.34431 &$ 10^{11}$&1782 &3.61513 \\
 \hline
$10^7$ & 61& 5.01138 &$ 10^{12}$&4663 &3.56008\\
 \hline
 $10^8$ &120 &4.07186 & $ 10^{13}$& 12260&3.47383 \\ 
 \hline
\end{tabular}
\end{center}
\caption{Statistics For Quadratic Twin Primes $p=n^2+1$ And $p+2$.} \label{t100}
\end{table}

\end{document}